\documentclass{amsart}
\usepackage{amssymb}
\usepackage{amscd}
\usepackage[all,cmtip]{xy}
\usepackage{hyperref}
\usepackage{graphicx}

\newcommand{\comment}[1]{}

\newcommand{\Q}{\mathbf{Q}}
\newcommand{\Z}{\mathbf{Z}}

\newcommand{\m}{\mathfrak{m}}
\newcommand{\pa}{\mathfrak{p}}

\newcommand{\bi}{\langle\,\, ,\,\rangle}

\newcommand{\limplies}{\Longleftarrow}

\theoremstyle{plain}
\newtheorem{theorem}{Theorem}[section]

\newtheorem{corollary}[theorem]{Corollary}

\newtheorem{lemma}[theorem]{Lemma}

\newtheorem{proposition}[theorem]{Proposition}

\theoremstyle{definition}
\newtheorem{definition}[theorem]{Definition}
\newtheorem{remark}[theorem]{Remark}

\newtheorem{examples}[theorem]{Examples}

\makeindex

\begin{document}

\begin{center}
{\large\bf Michiel Kosters - Universiteit Leiden \par}
{\large\bf \texttt{mkosters@math.leidenuniv.nl}, \today \par}
 \vspace{3em} {\LARGE\bf Anisotropic modules over artinian principal ideal rings \par} 
{\large\bf }
\end{center}

\section{Abstract}

Let $V$ be a finite-dimensional vector space over a field $k$ and let $W$ be a $1$-dimensional $k$-vector space. Let $\bi: V \times V \to W$ be a
symmetric bilinear form. Then $\bi$ is called \emph{anisotropic} if for all nonzero $v \in V$ we have $\langle v,v \rangle \neq 0$. Motivated by a
problem in algebraic number theory, we come up with a generalization of the concept of \emph{anisotropy} to symmetric bilinear forms on finitely
generated modules over artinian principal ideal rings. We will give many equivalent definitions of this concept of anisotropy. One of the
definitions shows that one can check if a form is anisotropic by checking if certain forms on vector spaces are
anisotropic. We will also discuss the concept of \emph{quasi-anisotropy} of a symmetric bilinear form, which has no vector space analogue. Finally we
will discuss the \emph{radical root} of a symmetric bilinear form, which doesn't have a vector space analogue either. All three concepts have
applications in algebraic number theory. 

\section{Introduction}

\subsection{The case of finite abelian groups}

Let $M$ be a finite abelian group of exponent $d$. Remark that $M$ is a finitely generated $\Z/d\Z$-module and that $\Z/d\Z$ is an artinian principal
ideal
ring. Using the structure theorem of finite abelian groups we can write 
\begin{eqnarray*}
M &=& \bigoplus_{p,i} F_{p,i}
\end{eqnarray*}
where $p$ ranges over the set of prime numbers, $i \in \Z_{>0}$ and the $F_{p,i}$ are free modules over $\Z/p^i \Z$. We call
$M$ \emph{semi-simple} if only the $F_{p,1}$ are nonzero, equivalently, if $d$ is squarefree. 
Let $\bi: M \times M \to \Q/\Z$ be a symmetric bilinear form. This form is called \emph{non-degenerate} if the map
\begin{eqnarray*}
M &\to& \mathrm{Hom}(M,\Q/\Z) \\
m &\mapsto& \left(m' \mapsto \langle m,m' \rangle \right)
\end{eqnarray*}
is an isomorphism. 

Now we define for all primes $p$ the $\Z/p\Z$-vector spaces
\begin{eqnarray*}
V_{p,\mathrm{odd}} = \bigoplus_{i \mathrm{\ odd}} F_{p,i}/pF_{p,i}
\end{eqnarray*}
and similarly
\begin{eqnarray*}
V_{p,\mathrm{even}} = \bigoplus_{i \mathrm{\ even}} F_{p,i}/pF_{p,i}.
\end{eqnarray*}
We define an inner product on $V_{p,\mathrm{odd}}$ as follows:
\begin{eqnarray*}
\bi_{p,\mathrm{odd}}: V_{p,\mathrm{odd}} \times V_{p,\mathrm{odd}} &\to& p^{-1}\Z/\Z \cong \Z/p\Z \\
\left( (x_i)_{i \mathrm{\ odd}}, (y_i)_{i \mathrm{\ odd}} \right) &\mapsto& \sum_{i \mathrm{\ odd}}  p^{i-1} \langle x_i, y_i \rangle,
\end{eqnarray*}
In the same way for $V_{p,\mathrm{even}}$ we obtain a form $\bi_{p,\mathrm{even}}$.
Later we
will even define these forms without using the structure theorem of finite abelian groups. We can now state a simplified version of the main theorem
of this article (Theorem \ref{main1}). 
For a subgroup $L \subseteq M$ we set $L^{\perp}=\{x \in M: \forall l \in L: \langle x,l \rangle=0 \}$. 

\begin{theorem} \label{sot}
Let $M$ be a finite abelian group and let $\bi: M \times M \to \Q/\Z$ be a symmetric bilinear form. For $p$ prime consider the $\Z/p\Z$-vector spaces
$V_{p, \mathrm{odd}}$ and $V_{p, \mathrm{even}}$ with their forms defined as above. 
Then the following statements are equivalent:
\begin{enumerate}
\item
for all primes $p$ the forms on $V_{p,\mathrm{odd}}$ and $V_{p,\mathrm{even}}$ are anisotropic as forms of $\Z/p\Z$-vector spaces; 
\item
$\bi$ is non-degenerate and there exists a unique submodule $L \subseteq M$ such that $L \subseteq L^{\perp}$ and $L^{\perp}/L$ is semi-simple. 
\end{enumerate}
\end{theorem}

In the second statement, the important part is the uniqueness. Indeed, consider the \emph{lower root} of $M$,
\begin{eqnarray*}
\mathrm{lr}(M)= \sum_{r \in \Z} \left( rM \cap M[r] \right)= \bigoplus_{p,i} p^{\lceil \frac{i}{2} \rceil} F_{p,i},
\end{eqnarray*}
where $M[r]=\{x \in M: rx=0\}$ and $rM=\{rm: m \in M\}$. If $\bi$ is non-degenerate, then we have $\mathrm{lr}(M) \subseteq \mathrm{lr}(M)^{\perp}$
and $\mathrm{lr}(M)^{\perp}/\mathrm{lr}(M)$ is semi-simple
(Corollary \ref{ji}). Hence in Theorem \ref{sot} ii it automatically follows that this unique submodule is $\mathrm{lr}(M)$. 

\begin{definition}
If one of the statements in the theorem above holds, then the form $\bi$ is called \emph{anisotropic}. 
\end{definition}

Remark that this definition is a generalization of the usual definition of anisotropy for a vector space over $\Z/p\Z$. Later in this article we will
discuss anisotropy in a more general setting.

Actually the naive generalization of the definition of anisotropy can partially be saved if $\# M$ is odd. Then in Theorem \ref{main1} we will show
that the form $\bi$ is anisotropic if and only if it is non-degenerate and any element $x \in M$ which satisfies $\langle x, x \rangle=0$ is an
element of $\mathrm{lr}(M)$. 

We have the following definitions.

\begin{definition}
For $p$ prime define the $\Z/p\Z$-vector space
\[V_{p,\mathrm{odd}}'=\bigoplus_{i \mathrm{\ odd},\ i \geq 3} F_{p,i}/pF_{p,i}\] with its natural form and recall the definition of
$V_{p,\mathrm{even}}$.  

Then $\bi$ is called \emph{quasi-anisotropic} if for all primes $p$ the forms on $V_{p,\mathrm{odd}}'$ and $V_{p,\mathrm{even}}$ are
anisotropic as forms of $\Z/p\Z$-vector spaces. 

We define the \emph{radical root} of $\bi$ to be
\begin{eqnarray*}
 \mathrm{rr}(M)=\bigcap_{L \subseteq M:\ L \subseteq L^{\perp},\ L^{\perp}/L \mathrm{\ semisimple}} L,
\end{eqnarray*}
where $L$ ranges over subgroups of $M$ and $L^{\perp}=\{x \in M: \forall l \in L: \langle x,l
\rangle=0\}$.
\end{definition}

In the article we will give some equivalent definitions of quasi-anisotropy and we will give a `formula' for the radical root. 

\subsection{Application in algebraic number theory}

We will now briefly discuss the applications in number theory. For the details read \cite{KO1} and \cite{KO2}. Given an order $A$ inside a number
field $K$, one uses the trace map from $K$ to $\Q$ to firstly describe a finite additive group $M$ of size equal to the absolute value of the
discriminant of $A$ and secondly a $\Q/\Z$-valued non-degenerate symmetric bilinear form on $M$. The ring of integers, $\mathcal{O}_K$, corresponds,
under some tameness assumptions, to a subgroup $L$ of $M$ that satisfies $L \subseteq L^{\perp}$ and $L^{\perp}/L$ is semi-simple. In the case of
anisotropy one now directly obtains that the group corresponding to $\mathcal{O}_K$ is equal to $\mathrm{lr}(M)$ and one can find $\mathcal{O}_K$
directly.  

It turns out that for our applications in number theory, the weaker statement of quasi-anisotropy gives similar
results. The radical root in many cases gives a large part of the ring of integers, and hence is also important. 

\section{Artinian principal ideal rings}

We will now generalize and prove the statements from the introduction. Let $R$ be an artinian principal ideal ring and let $M$ be a finitely generated
$R$-module. The following lemma allows us to give local proofs in many cases.

\begin{lemma} \label{ds}
The following statements hold.
\begin{enumerate}
\item
$R$ has only finitely many prime ideals;
\item
$R=\prod_{\pa \in \mathrm{Spec}(R)} R_{\pa}$; 
\item
$M=\bigoplus_{\pa \in \mathrm{Spec}(R)} M_{\pa}$;
\item
$\mathrm{Hom}_R(M,R)=\bigoplus_{\pa \in \mathrm{Spec}(R)} \mathrm{Hom}_{R_{\pa}}(M_{\pa},R_{\pa})$.
\end{enumerate}
\end{lemma}
\begin{proof}
The first statement follows from \cite{AT}, Theorem 8.5 and Proposition 8.3. For the second statement see \cite{LA}, Exercise 10.9f. The third
statement follows from the second one by tensoring with $M$ over $R$. The last statement now follows easily. 
\end{proof}

Assume for the moment that $R$ is local as well and let $\m$ be its maximal ideal. For example one can take $R$ to be a field or $\Z/p^n\Z$ where $p$
is prime and $n \in \Z_{\geq 1}$. Fix a generator $\pi$ of the maximal ideal. As $R$ is local and artinian, there is a smallest $n \in \Z_{\geq 1}$
such that $\m^n=0$ (see \cite{AT}, Proposition 8.4). We fix this number $n$ from now on. One can show that all ideals of $R$ occur in the series $R
\supsetneq \m \supsetneq \m^2 \supsetneq \ldots \supsetneq \m^n=0$ and that $R$ has length $n$ as an $R$-module. We will write any ideal of
$R$ as $\m^i$ where $0 \leq i \leq n$.

We can now easily prove the following result. 

\begin{corollary} \label{inj}
The ring $R$ is injective as an $R$-module.
\end{corollary}
\begin{proof}
According Baer's criterion (\cite{LAM}, 3.7) $R$ is injective if and only if we can extend any $R$-linear map $f: I \to R$ where $I$ is an ideal
of $R$ to a map $f': R \to R$. Using the above remarks about $\mathrm{Hom}$ we may assume that $R$ is local. We know all ideals explicitly and an
easy calculation shows that we can extend such maps. 
\end{proof}

We have the following structure theorem for modules over $R$. For any prime $\pa \subset R$ let $n_{\pa}=\mathrm{length}_{R_{\pa}}(R_{\pa})$.

\begin{theorem} \label{prod}
We have $M \cong \bigoplus_{\pa \in \mathrm{Spec}(R)} \bigoplus_{i=1}^{n_{\pa}} \left( R/\pa^i \right)^{n_{\pa,i}}$ for certain unique $n_{\pa,i} \in \Z_{\geq 0}$. 
\end{theorem}
\begin{proof}
Use Lemma \ref{ds} to reduce to the local case and notice that $R_{\pa}/\pa^i R_{\pa} \cong R/\pa^i$ using exactness of localization. Now use
\cite{KO1}, Theorem 3.3.3 to obtain the result (we actually obtain a stronger result). 
\end{proof}

Recall that $M$ is called \emph{semi-simple} if it is a direct sum of simple submodules. That means in the previous theorem that $M$ is semi-simple iff for all $\pa \in \mathrm{Spec}(R)$ and $i \neq 1$ we have $n_{\pa,i}=0$. 

\begin{corollary} \label{hom}
We have $M \cong_R \mathrm{Hom}_R(M,R)$.
\end{corollary}
\begin{proof}
Using Lemma \ref{ds} we may assume that $(R,\m)$ is local. By the previous theorem, we only have to check it for the
modules of the form $M=R/\m^i$ for some $i \in \Z_{\geq 1}$. This calculation is left to the reader. 
\end{proof}

We now define the upper and lower root of $M$. Recall that for $r \in R$ we define
\begin{eqnarray*}
M[r] &=& \{ m \in M: rm=0 \}, \\
rM &=& \{rm: m \in M \}.
\end{eqnarray*}

\begin{definition}
We define the \emph{lower root} of $M$ as
\begin{eqnarray*}
\mathrm{lr}(M)= \sum_{r \in R} \left( rM \cap M[r] \right). 
\end{eqnarray*}
We define the \emph{upper root} of $M$ as
\begin{eqnarray*}
\mathrm{ur}(M)= \bigcap_{r \in R} \left( rM + M[r] \right). 
\end{eqnarray*}
\end{definition}

\begin{remark}
Let $M'$ be another finitely generated $R$-module. Then $\mathrm{lr}(M \oplus M')=\mathrm{lr}(M) \oplus \mathrm{lr}(M')$. A similar statement holds for the upper root. 
\end{remark}

\begin{lemma} \label{loup}
The following statements hold.
\begin{enumerate}
\item
Let $\pa \in \mathrm{Spec}(R)$ and let $M=R/\pa^i$ where $0 \leq i \leq n_{\pa}$. Then $\mathrm{lr}(M)=\pa^{\lceil \frac{i}{2} \rceil}/\pa^i$ and $\mathrm{ur}(M)=\pa^{\lfloor \frac{i}{2} \rfloor}/\pa^i$. Stated differently, the lower root respectively upper root of such a cyclic module of length $i$ is the unique submodule of length $\lfloor \frac{i}{2} \rfloor$ respectively $\lceil \frac{i}{2} \rceil$.
\item
We have $\mathrm{lr}(M) \subseteq \mathrm{ur}(M)$ and $\mathrm{ur}(M)/\mathrm{lr}(M)$ is semi-simple.
\end{enumerate}
\end{lemma}
\begin{proof}
The first statement is a calculation which is left to the reader and the second statement now follows from Theorem \ref{prod}. 
\end{proof}

\begin{remark}
From Theorem \ref{prod} and Lemma \ref{loup} we see that $\mathrm{lr}(M)=0$ $\iff$ $M$ is semi-simple. 
\end{remark}

\section{Bilinear forms}

From now on let $N$ be an $R$-module such that $N \cong_R R$ and let $\bi: M \times M \to N$ be a symmetric $R$-bilinear form where $R$ is
an artinian principal ideal ring. 

\begin{definition}
The form $\bi$ is called \emph{non-degenerate} if the map
\begin{eqnarray*}
M &\to& \mathrm{Hom}_R(M,N) \\
m &\mapsto& \left( m' \mapsto \langle m,m' \rangle \right)
\end{eqnarray*}
is an isomorphism. 
\end{definition}

Recall that for a submodule $M' \subseteq M$ we define $M'^{\perp}=\{ a \in M: \langle a, M' \rangle=0 \}$, which is an $R$-submodule of $M$. We say
that $M=M_1 \perp M_2$ if $M=M_1 \oplus M_2$ and $M_1 \subseteq M_2^{\perp}$. 

\begin{lemma} \label{orthogonal} \label{sir0}
Let $M'$ be an $R$-module and let $\bi': M'\times M' \to N$ be a non-degenerate symmetric bilinear form. Let $\varphi: M' \to M$ be an $R$-linear map
which respects the symmetric bilinear forms. Then $\varphi$ is injective and viewing $M'$ as a submodule of $M$ we have $M= M' \perp M'^{\perp}$ with
respect to $\langle\,\, ,\,\rangle$. Furthermore, $\langle\,\, ,\,\rangle$ is non-degenerate iff $M'^{\perp}$ is non-degenerate.
\end{lemma}
\begin{proof}
It is very easy to see that $\varphi$ is injective. See \cite{KO1}, Theorem 1.1.8 for the proof of the rest.
\end{proof}

We have the following useful lemma.

\begin{lemma} \label{ds2}
We have
\begin{eqnarray*}
M &=& \perp_{\pa \in \mathrm{Spec}(R)} M_{\pa}.
\end{eqnarray*}
and the form $\bi$ naturally induces for any prime $\pa \in \mathrm{Spec}(R)$ a form 
\begin{eqnarray*}
\bi_{\pa}: M_{\pa} \times M_{\pa} &\to& N_{\pa} \\
(\frac{x}{s},\frac{y}{t}) &\mapsto& \frac{ \langle x,y \rangle }{st}.
\end{eqnarray*}
The form $\bi$ is non-degenerate iff all the $\bi_{\pa}$ are non-degenerate as symmetric $R_{\pa}$-bilinear forms. 
\end{lemma}
\begin{proof}
The proof is left to the reader. 
\end{proof}

\begin{theorem} \label{ar1}
Consider the map 
\begin{eqnarray*}
^{\perp}: \{ R\mathrm{-submodules\ of\ }M\} &\to& \{ R\mathrm{-submodules\ of\ }M\} \\
M' &\mapsto& M'^{\perp}.
\end{eqnarray*}
Then the following statements are equivalent:
\begin{enumerate}
\item
$\langle\,\, ,\,\rangle$ is non-degenerate;
\item
$M^{\perp}=0$;
\item
for all submodules $M' \subseteq M$ we have \[\mathrm{length}_{R}(M')+\mathrm{length}_R(M'^{\perp})=\mathrm{length}_R(M); \]
\item
$^{\perp}$ is an inclusion reversing bijection with inverse $^{\perp}$.
\end{enumerate}
\end{theorem}

\begin{proof}
We may assume that $N=R$.

i $\implies$ iii: 
Let $M' \subseteq M$ be a submodule. Let $\varphi: M \to \mathrm{Hom}_R(M,R)$ be the isomorphism obtained from $\langle\,\, ,\,\rangle$. By injectivity of $R$ (Corollary \ref{inj}) we know that $0 \to \mathrm{Hom}_R(M/M',R) \to \mathrm{Hom}_R(M,R) \to \mathrm{Hom}_R(M',R) \to 0$ is exact. By definition $M'^{\perp}=\varphi^{-1}(\mathrm{Hom}_R(M/M',R))$. We find that $\mathrm{length}_R(M')+\mathrm{length}_R(M'^{\perp})=\mathrm{length}_R(M)$ by Corollary \ref{hom}.

iii $\implies$ iv: 
Let $M' \subseteq M$ be a submodule. Then we directly have $M' \subseteq \left( M'^{\perp} \right)^{\perp}$. By iii both have the same length and we have an equality as required. 

iv $\implies$ ii: 
We have
\begin{eqnarray*}
M^{\perp}=(0^{\perp})^{\perp}=0.
\end{eqnarray*}

ii $\implies$ i: 
We obtain a morphism $\varphi: M \to \mathrm{Hom}_R(M,R)$ from $\langle\,\, ,\,\rangle$. By assumption, $\varphi$ is injective. As $\mathrm{length}_R(\mathrm{Hom}_R(M,R))=\mathrm{length}_R(M)$ (Corollary \ref{hom}), it follows that the map is surjective as well and we have an isomorphism.

\end{proof}

\begin{corollary} \label{per}
Assume that $\bi$ is non-degenerate. Then for all submodules $M', M'' \subseteq M$ we have
\begin{eqnarray*}
\left( M' + M'' \right)^{\perp} &=& M'^{\perp} \cap M''^{\perp} \\
\left( M' \cap M''\right)^{\perp} &=& M'^{\perp}+M''^{\perp}.
\end{eqnarray*}
\end{corollary}
\begin{proof}
Both properties follow from Theorem \ref{ar1} iv. 
\end{proof}

\begin{corollary} \label{ji}
We have $\mathrm{lr}(M) \subseteq \mathrm{ur}(M)^{\perp} \subseteq \mathrm{lr}(M)^{\perp}$. If $\bi$ is non-degenerate, we have $\mathrm{lr}(M)^{\perp}=\mathrm{ur}(M)$, $\mathrm{lr}(M) \subseteq \mathrm{lr}(M)^{\perp}$ and $\mathrm{lr}(M)^{\perp}/\mathrm{lr}(M)$ is semi-simple
\end{corollary}
\begin{proof}
Using Lemma \ref{ds2} we may assume that $(R,\m)$ is local. The first statement follows easily from the definitions. That
$\mathrm{lr}(M)^{\perp}=\mathrm{ur}(M)$ follows from the observation that $\left( \m^k M\right)^{\perp}=M[\m^k]$ in the
non-degenerate case and Corollary \ref{per}. We then directly have that $\mathrm{lr}(M) \subseteq \mathrm{lr}(M)^{\perp}$. The last statement now
follows by using Lemma \ref{loup} ii.
\end{proof}

\begin{lemma} \label{hoa}
For any $\pa \in \mathrm{Spec}(R)$ the form $\bi_{\pa}$ induces a symmetric $R/\pa$-bilinear form
\begin{eqnarray*}
\langle\,\, ,\,\rangle_{\pa,\mathrm{odd}}: \mathrm{ur}(M_{\pa})/\mathrm{lr}(M_{\pa}) \times \mathrm{ur}(M_{\pa})/\mathrm{lr}(M_{\pa}) &\to& N_{\pa}[\pa] \\
\left([x],[y] \right) &\mapsto& \langle x,y \rangle_{\pa}.
\end{eqnarray*} 
If $\langle\,\, ,\,\rangle$ is non-degenerate, then $\langle\,\, ,\,\rangle_{\pa, \mathrm{odd}}$ is non-degenerate as well. 
\end{lemma}
\begin{proof}
The first statement follows from Corollary \ref{ji} and Lemma \ref{ds2}. The second statement follows from Corollary \ref{ji}, Theorem \ref{ar1} and
Lemma \ref{ds2}. 
\end{proof}

Notice that $N_{\pa}[\pa] \cong R_{\pa}[\pa] \cong R/\pa$, and $\mathrm{ur}(M_{\pa})/\mathrm{lr}(M_{\pa})$ is a vector space over $R/\pa$. 

\begin{definition} \label{odd}
Let $\pa \in \mathrm{Spec}(R)$. We define $\langle\,\, ,\,\rangle_{\pa,\mathrm{odd}}$ to be the form obtained from $\langle\,\, ,\,\rangle$ on
$\mathrm{ur}(M_{\pa})/\mathrm{lr}(M_{\pa})$ as in Lemma \ref{hoa}. From $\langle\,\, ,\,\rangle_{\pa}$ we obtain a form $\langle\,\,
,\,\rangle_{\pa}': M_{\pa}/M_{\pa}[\pa] \times M_{\pa}/M_{\pa}[\pa] \to N_{\pa}/N_{\pa}[\pa]$. This form is a non-degenerate symmetric
$R_{\pa}/R_{\pa}[\pa]$-bilinear form if $\bi$ is non-degenerate over $R$ by Lemma \ref{hoa} and Theorem \ref{ar1}. We define $\langle\,\,
,\,\rangle_{\pa,\mathrm{even}}=\left( \langle\,\, ,\,\rangle'_{\pa} \right)_{\mathrm{odd}}$. We will often use the notation $M_{\pa,\mathrm{odd}}$ and
$M_{\pa,\mathrm{even}}$ for these modules with their symmetric bilinear forms. Remark that the odd and even forms for a given prime $\pa$ are forms on
vector spaces over $R/\pa$.    
\end{definition}

\begin{remark} \label{rep}
In practice one can easily calculate the odd and even forms if one has the decomposition of $M$ as in Theorem \ref{prod}. This is what we did in the
introduction with the $V_{p,\mathrm{even}}$ and $V_{p,\mathrm{odd}}$. The main difference between the exposition here and the one in the introduction
is that here we don't have to choose a generator of the maximal ideal of $R_{\pa}$.   
\end{remark}

\section{Equivalent definitions of anisotropy}

In this section we keep the assumptions and notations of the previous section. 

\begin{definition}
The form $\bi$ is called \emph{anisotropic} if for all primes $\pa \in \mathrm{Spec}(R)$ the forms $\bi_{\pa, \mathrm{even}}$ and $\bi_{\pa,
\mathrm{odd}}$ are anisotropic in the usual sense as forms over $R/\pa$ (as in the abstract). 
\end{definition}

\begin{remark}
This definition of anisotropy looks a bit weird at first sight, but there is a striking resemblance with the following Proposition 1.9 on page 147 of 
\cite{LAM2}: Suppose $F$ is a non-archimedian field with valuation $v$ with $\pi$ as a uniformizer with residue
characteristic not equal to $2$. Let $U=\{x \in F: v(x)=0\}$. Suppose we have a form $q=q_1 \perp \langle \pi \rangle q_2$ where $q_1=\langle
u_1,\ldots,u_r \rangle$, $q_2=\langle u_{r+1},\ldots,u_n\rangle$ where $u_i \in U$. Then $q$ is anisotropic over $F$ if and only if the reduction of
$q_1$ and $q_2$ to the residue field are anisotropic. 

We didn't investigate this resemblance in more detail.
\end{remark}

\begin{remark}
If $R$ is a field, then $\bi_{0,\mathrm{even}}=0$ and $\bi_{0,\mathrm{odd}}=\bi$, hence the above definition is a generalization of the normal concept
of anisotropy for forms on vector spaces. 
\end{remark}

\begin{remark}
As a non-degenerate symmetric form on a $1$-dimensional vector space is automatically anisotropic, we deduce the following. If $\bi$ is non-degenerate
and $M$ is a cyclic $R$-module, it is anisotropic. Also if $R$ is local, $\bi$ is non-degenerate, $M$ is generated by two elements and $M$ has odd
length as an $R$-module, then $\bi$ is anisotropic. 
\end{remark}

We can now state the following `equivalent' definitions of anisotropy. The proof will be given in Section \ref{pr}.

\begin{theorem} \label{main1}
Let $R$ be an artinian principal ideal ring, let $M$ be a finitely generated $R$-module and let $N$ be an $R$-module such that $N \cong_R R$.
Let $\langle\,\, ,\,\rangle: M
\times M \to N$ be a symmetric $R$-bilinear form. Consider the following statements:
\begin{enumerate}
\item
$\langle\,\, ,\,\rangle$ is anisotropic;
\item
the form $\bi$ is non-degenerate and $\mathrm{lr}(M)$ is the only submodule $L \subseteq M$ satisfying $L \subseteq L^{\perp}$ such that $L^{\perp}/L$
is semi-simple;
\item
the form $\bi$ is non-degenerate and for any submodule $L \subseteq M$ with $L \subseteq L^{\perp}$, we have $L \subseteq \mathrm{lr}(M)$ and $\mathrm{lr}(L^{\perp}/L)=\mathrm{lr}(M)/L$; 
\item 
the form $\langle\,\, ,\,\rangle$ is non-degenerate and for any submodule $L \subseteq M$ with $L \subseteq L^{\perp}$, we have $L \subseteq \mathrm{lr}(M)$; 
\item
the form $\langle\,\, ,\,\rangle$ is non-degenerate and if $x \in M$ satisfies $\langle x,x \rangle =0$, then $x \in \mathrm{lr}(M)$. 
\end{enumerate} 
Then \emph{i} $ \Longleftrightarrow$ \emph{ii} $\Longleftrightarrow$ \emph{iii} $\implies$ \emph{iv} $\Longleftrightarrow$ \emph{v}. If $2 \in \left(R/\mathrm{Ann}_R(M) \right)^*$, then all statements are equivalent. 
\end{theorem}

\begin{remark}
 Taking into account Remark \ref{rep}, one sees that Theorem \ref{sot} is a special case of Theorem \ref{main1}.
\end{remark}

\begin{examples}
Let $R=N=\Z/2^2\Z$. Consider the non-degenerate bilinear form $\langle\,\, ,\,\rangle$ on $M=\Z/2^2\Z \times \Z/2^2\Z$ given by the following matrix:
\begin{eqnarray*}
A= \left( \begin{array}{cc}
2 & 1 \\
1 & 2
\end{array} \right).
\end{eqnarray*}
Then one can show that iv and v hold but i, ii and iii don't hold. 
\end{examples}

\section{The odd and even forms}

In this section we assume that $(R,\m)$ is a local artinian principal ideal ring. Let $n$ be the
length of $R$ as an $R$-module, and let $\m=(\pi)$. 

Let $M$ be a finitely generated $R$-module, $N$ an $R$-module such that $N \cong_R R$ and let $\bi: M \times M \to N$ be a symmetric $R$-bilinear
form. To make the notation less heavy, we use the notation $M_{\mathrm{odd}}$ instead of
$M_{\m,\mathrm{odd}}$ and similarly for other notation that refers to the unique prime ideal $\m$ of $R$.  

\subsection{Splitting up}

We first want to have more control over $M_{\mathrm{odd}}$ and $M_{\mathrm{even}}$ and this is why we divide them up into smaller parts. 

\begin{definition}
Let $i \in \Z_{\geq 1}$ be odd. Then we define \index{$\mathrm{\rho_i(M)}$}
\begin{eqnarray*}
\rho_i(M) &=& \m^{\lfloor \frac{i}{2} \rfloor} M[\m^i]/ \left( \m^{\lfloor \frac{i}{2} \rfloor } M[\m^{i-1}]+\m^{\lfloor \frac{i}{2} \rfloor +1}M[\m^{i+1}] \right).
\end{eqnarray*}
For even $i \in \Z_{\geq 2}$ we define 
\begin{eqnarray*}
\rho_i(M) &=& \rho_{i-1} (M/M[\m]).
\end{eqnarray*}
\end{definition}

By construction these $\rho_i(M)$ are vector spaces over $R/\m$. In fact, we can view $\rho_i$ as a functor from the category of $R$-modules to the
category of $R/\m$-modules. Also notice that the functor is additive. 	

\begin{lemma} \label{soe}
The natural map 
\begin{eqnarray*}
\varphi_{\mathrm{odd}}: \bigoplus_{i \in \Z_{\geq 1} \mathrm{\ odd}} \rho_i(M) &\to& \mathrm{ur}(M)/\mathrm{lr}(M)=M_{\mathrm{odd}} \\
~([x_i])_{i \mathrm{\ odd}} &\mapsto& [\sum_{i \mathrm{\ odd}} x_i] 
\end{eqnarray*}
is an isomorphism of $R$-modules. Similarly, we obtain an isomorphism of $R$-modules
\begin{eqnarray*}
\varphi_{\mathrm{even}}: \bigoplus_{i \in \Z_{\geq 2} \mathrm{\ even}} \rho_i(M) &\to& \mathrm{ur}(M/M[\m])/\mathrm{lr}(M /M[\m]) =M_{\mathrm{even}} \\
~([x_i])_{i \mathrm{\ even}} &\mapsto& [\sum_{i \mathrm{\ even}} x_i]. 
\end{eqnarray*}
Now assume that $M \cong \left( R/\m^r \right)^s$ (where $s \in \Z_{\geq 0}$, $1 \leq r \leq n$). Then we have an $R$-linear isomorphism
\begin{eqnarray*}
\rho_i(M) \cong \left\{
\begin{array}{cc}
\left(R/\m \right)^s & i=r \\
0 & i \neq r 
\end{array}
\right. 
\end{eqnarray*}

\end{lemma}
\begin{proof}
We sketch a proof of the statements for $\varphi_{\mathrm{odd}}$ and by construction the similar statements for $\varphi_{\mathrm{even}}$ will follow.
The reader can check that $\varphi_{\mathrm{odd}}$ is well defined. To check that the map is a bijection, it suffices by Theorem \ref{prod} to check
it for modules of the form
$R/\m^r$ for $r \in \Z_{\geq 1}$, and this is left to the reader.
\end{proof}

\begin{definition} \label{goh}
We define the symmetric $R/\m$-bilinear form $\langle\,\, ,\,\rangle_{\mathrm{odd}}'$
\index{$\langle\,\, ,\,\rangle_{\mathrm{odd}}'$} to be the map making the following diagram
commute:
\[
\xymatrix{ 
\bigoplus_{i \in \Z_{\geq 1} \mathrm{\ odd}} \rho_i(M) \times \bigoplus_{i \in \Z_{\geq 1} \mathrm{\ odd}} \rho_i(M) \ar[d]^{\varphi_{\mathrm{odd}}
\times \varphi_{\mathrm{odd}}}  \ar[rrrr]^{\langle\,\, ,\,\,\rangle_{\mathrm{odd}}'} &&&& N[\m]  \ar[d]_{\mathrm{id}_{N[\m]}} \\
M_{\mathrm{odd}} \times M_{\mathrm{odd}} \ar[rrrr]_{\langle\,\, ,\,\,\rangle_{\mathrm{odd}}} &&&& N[\m].
}
\] 
We define the symmetric $R/\m$- ($n \geq 2$) respectively $0$-bilinear form ($n<2$) $\langle\,\,
,\,\rangle_{\mathrm{even}}'$ \index{$\langle\,\, ,\,\rangle_{\mathrm{even}}'$}
\[
\xymatrix{ 
\bigoplus_{i \in \Z_{\geq 2} \mathrm{\ even}} \rho_i(M) \times \bigoplus_{i \in \Z_{\geq 2} \mathrm{\ even}} \rho_i(M) \ar[d]^{\varphi_{\mathrm{even}}
\times \varphi_{\mathrm{even}}}  \ar[rrr]^{\ \ \ \ \ \ \ \ \langle\,\, ,\,\,\rangle_{\mathrm{even}}'} &&& N[\m^2]/N[\m] 
\ar[d]_{\mathrm{id}_{N[\m^2]/N[\m]}} \\
M_{\mathrm{even}} \times M_{\mathrm{even}} \ar[rrr]_{\ \ \ \ \ \ \ \ \langle\,\, ,\,\,\rangle_{\mathrm{even}}} &&& N[\m^2]/N[\m].
}
\]
\end{definition}

\begin{lemma} \label{merk}
With respect to $\bi'_{\mathrm{odd}}$ respectively $\bi'_{\mathrm{even}}$ we have orthogonal decompositions 
\begin{eqnarray*}
\perp_{i \in \Z_{\geq 1} \mathrm{\ odd}} \rho_i(M)
\end{eqnarray*} 
respectively
\begin{eqnarray*}
\perp_{i \in \Z_{\geq 2} \mathrm{\ even}} \rho_i(M).
\end{eqnarray*} 
\end{lemma}
\begin{proof}
This is an easy calculation, which is left to the reader.
\end{proof}

\begin{lemma} \label{homc}
Let $M=(R/\m^r)^s$ ($1 \leq r \leq n$, $s \in \Z_{\geq 0}$). Say that $e_1,\ldots,e_s$ form a basis of $M$ over $R/\m^r$. Let $x_{ij} \in R$ with
$\langle e_i,e_j \rangle=\pi^{n-r}x_{ij}$. Then $\langle\,\, ,\,\rangle$ is non-degenerate iff $\det{\left((x_{ij}+\m)_{i,j=1}^s \right)} \neq
0$ in $R/\m$. 
\end{lemma}
\begin{proof}
This follows directly from the definition of non-degeneracy. 
\end{proof}

\begin{theorem} \label{rete} \label{moma}
The following statements hold.
\begin{enumerate}
\item
The form $\langle\,\, ,\,\rangle$ is non-degenerate iff $\langle\,\, ,\,\rangle_{\mathrm{odd}}$ and $\langle\,\, ,\,\rangle_{\mathrm{even}}$ are non-degenerate. 
\item
Assume that $\langle\,\, ,\,\rangle$ is non-degenerate. Then we can write $M=M_1 \perp \ldots \perp M_n$ where $M_i$ is non-degenerate and free over
$R/\m^i$. Fix $j$ ($1 \leq j \leq n$) and assume that $M_j' \subseteq M$ is free over $R/\m^j$ and that the map $\rho_j(M_j') \to \rho_j(M)$ is an
isomorphism of $R$-modules. Then there is a decomposition as above with $M_j=M_j'$.    
\end{enumerate}
\end{theorem}
\begin{proof}
We will prove the first statement, and along the way we will prove the second one as well. 

\noindent$\implies$\hspace{-0.1cm}: From Lemma \ref{hoa} it follows that $\langle\,\, ,\,\rangle_{\mathrm{odd}}$ is non-degenerate. For the
even case, we notice that the form $M/M[\m] \times M/M[\m] \to N/N[\m]$ is still non-degenerate, and one can then apply Lemma
\ref{hoa}.

\noindent$\limplies$: First assume that $M$ is free over $R/\m^r$ ($1 \leq r \leq n$), say $M=(R/\m^r)^s$ with basis $e_1,\ldots, e_s$ over $R/\m^r$.
Assume
that $r$ is odd. Then we write
\begin{eqnarray*}
\langle e_i,e_j \rangle &=& \pi^{n-r} x_{ij}
\end{eqnarray*}
where $x_{ij} \in R$. Notice that the elements $\pi^{\frac{r-1}{2}}e_i$ give a basis of $\mathrm{ur}(M)/\mathrm{lr}(M)$. We obtain
\begin{eqnarray*}
\pi^{r-1} \langle e_i,e_j \rangle &=& \langle \pi^{\frac{r-1}{2}}e_i,\pi^{\frac{r-1}{2}}e_j \rangle \\
&=& \pi^{n-1} x_{ij}.	 
\end{eqnarray*}
By Lemma \ref{homc} we see $\langle\,\, ,\,\rangle$ is non-degenerate iff $\langle\,\, ,\,\rangle_{\mathrm{odd}}$ is non-degenerate.

Similarly, one shows that if $r$ is even, $\langle\,\, ,\,\rangle$ is non-degenerate iff $\langle\,\, ,\,\rangle_{\mathrm{even}}$ is non-degenerate. 
 
Now we will do the general case. We will give a proof by induction on $\mathrm{length}_R(M)$, the statement being trivial when $M=0$. Pick a nonzero
$M_i \subseteq M$ free over $R/\m^i$ ($1 \leq i \leq n$) such that the map $\rho_i(M_i') \to \rho_i(M)$ is an isomorphism (which we can do by Theorem
\ref{prod} and Lemma \ref{soe}). Then from the homogeneous case and Lemma \ref{merk} it follows that $M_i$ is non-degenerate. Now write $M=M_i \perp
M'$ (Lemma
\ref{orthogonal}) and by induction $M'$ is non-degenerate. In the last step we used that the odd and even forms stay non-degenerate by Lemma
\ref{soe}. This finishes the proof of the second implication and of ii.  
\end{proof}

\subsection{Shaving}
From now on let $r$ be minimal such that $\mathrm{Ann}_R(M)=\m^r$. In many proofs we want to do induction on $r$ and for this purpose we use a
technique called shaving. 

\begin{definition}
Suppose that $r \geq 2$. Then we define \index{$\mathrm{Sh}(M)$}  \index{shaving}
\begin{eqnarray*}
\mathrm{Sh}(M)=M[\m^{r-1}]/\m^{r-1}M.
\end{eqnarray*}
\end{definition}

Notice that $\mathrm{Sh}(M)$ obtains a natural symmetric $R$-bilinear form from $\bi$ as $\langle M[\m^{r-1}], \m^{r-1}M \rangle=0$. So assume that
$\mathrm{Sh}(M)$ is equipped with this natural form. 

\begin{lemma} \label{trivi}
Suppose that $r \geq 2$. Let $\varphi: M[\m^{r-1}] \to \mathrm{Sh}(M)$ be the canonical map. Then:
\begin{eqnarray*}
\mathrm{lr}(M) &=& \varphi^{-1}\left(\mathrm{lr}(\mathrm{Sh}(M))\right) \\
\mathrm{ur}(M) &=& \varphi^{-1}\left(\mathrm{ur}(\mathrm{Sh}(M))\right).
\end{eqnarray*}
\end{lemma}
\begin{proof}
This is an easy calculation and left to the reader. 
\end{proof}

\begin{lemma} \label{threes}
Suppose that $r \geq 2$. Then we have natural isomorphisms for $i \geq 1$ that respect the inner products given by Definition \ref{goh}
\begin{eqnarray*}
\rho_i(\mathrm{Sh}(M)) &\cong& \left\{
\begin{array}{cc}
0 & i \geq r \\
\rho_{r-2}(M) \perp \rho_r(M)& i=r-2, r \geq 3, \\
\rho_{i}(M) & i=r-1 \mathrm{\ or\ } i<r-2
\end{array} \right.
\end{eqnarray*}
\end{lemma}
\begin{proof}
For the proof, choose a decomposition of $M$ into homogeneous modules, that is, modules which are free over some $R/\m^i$, and use Lemma \ref{soe}. 
\end{proof} 

\begin{remark}
 Lemma \ref{threes} shows that if $r \geq 3$ we don't lose any information about the $\rho_i$ and their forms if we pass from $M$ to $\mathrm{Sh}(M)$.
In this case we see that $M$ is
anisotropic if and only if $\mathrm{Sh}(M)$ is anisotropic. Furthermore, by Theorem \ref{moma}, we see that if $r \geq 3$, a non-degenerate form
induces a non-degenerate form on $\mathrm{Sh}(M)$. 
\end{remark}

\begin{lemma} \label{sat}
Assume that $r \geq 2$ and that $\bi$ is non-degenerate. Let $\varphi: M[\m^{r-1}] \to M[\m^{r-1}]/\m^{r-1}M$ be the natural map. Let
\begin{eqnarray*}
\mathfrak{S}_1 &=& \{ L \subseteq M: \m L^{\perp} \subseteq L \subseteq L^{\perp}, L \subseteq M[\m^{r-1}] \} \\
\mathfrak{S}_2 &=& \{ L' \subseteq \mathrm{Sh}(M): \m L'^{\perp} \subseteq L' \subseteq L'^{\perp} \}.
\end{eqnarray*}
Then we have the following surjection
\begin{eqnarray*}
\psi: \mathfrak{S}_1 &\to& \mathfrak{S}_2 \\
L &\mapsto& \varphi(L).
\end{eqnarray*}
If we restrict the domain to the set of all $L \subset M$ that also satisfy $\m^{r-1}M \subseteq L$, then the map is a bijection.
\end{lemma}
\begin{proof}
To see that the map is well defined, it suffices to show that if $L \in \mathfrak{S}_1$, then $L'= L+\m^{r-1} M$ satisfies $\m L'^{\perp} \subseteq L'
\subseteq L'^{\perp}$. This is an easy calculation, which doesn't require the non-degeneracy. 
Now suppose that $L' \subseteq \mathrm{Sh}(M)$ satisfies $\m L'^{\perp} \subseteq L' \subseteq L'^{\perp}$. Then $L=\varphi^{-1}(L')$ satisfies $\m L^{\perp} \subseteq L \subseteq L^{\perp}$ (here we use that $L^{\perp} \subseteq M[\m^{r-1}]$ by non-degeneracy). As $\m^{r-1}M \subseteq L$ and $\varphi(L)=L'$  we find a bijection if we restrict our domain. 
\end{proof}

\section{Preparation for the proof of the equivalence}

In this section we will prove many small lemmas that are used in the proof of Theorem \ref{main1}. Assume that we are in the same situation as in
the previous section. We begin with the definition of the radical root.

\begin{definition} \label{rad}
We define the \emph{radical root} of $(M,\bi)$ as 
\begin{eqnarray*}
\mathrm{rr}(M,\bi)=\mathrm{rr}(M)=\bigcap_{L \subseteq M: L \subseteq L^{\perp},\ L^{\perp}/L \mathrm{\ semisimple}} L.
\end{eqnarray*}
\end{definition}

We will come back to the radical root in the last section of this article.

\begin{lemma} \label{sir3}
Assume that $\bi$ is non-degenerate and that $r \geq 2$. Suppose that $\rho_r(M)$ is anisotropic. Then $\m^{r-1}M \subseteq \mathrm{rr}(M)$. 
\end{lemma}
\begin{proof}

Take $L \subseteq M$ with $\m L^{\perp} \subseteq L \subseteq L^{\perp}$. Suppose $\m^{r-1}M \not \subseteq L$. Then $L^{\perp} \not \subseteq M[\m^{r-1}]$ (by Theorem \ref{ar1}). Let $x \in L^{\perp} \setminus M[\m^{r-1}]$. 

First suppose that $r$ is odd. Then take $y=\pi^{\lfloor \frac{r}{2} \rfloor}x$. Then $0 \neq [y] \in \rho_r(M)$, which follows as $x \not \in
M[\m^{r-1}]$. As $\langle x,\pi x \rangle=0$ (since $\pi x \in L$), it follows that $\langle y,y \rangle=0$, as $r \geq 2$. This shows that
$\rho_r(M)$ is isotropic, a contradiction.

Now suppose that $r$ is even and consider $y=\pi^{\lfloor \frac{r-1}{2} \rfloor}x$. We see that $\pi \langle y,y \rangle=0$, but $[y] \neq 0$ in
$\rho_r(M)$ as $x \not \in M[\m^{r-1}]$. Hence $\rho_r(M)$ is isotropic, a contradiction. 
\end{proof}

\begin{lemma} \label{killit}
Let $L \subseteq M$ satisfy $L \subseteq L^{\perp}$. Then any maximal submodule $L' \supseteq L$ with $L' \subseteq L'^{\perp}$ satisfies $\m L'^{\perp} \subseteq L'$.  
\end{lemma}
\begin{proof}
Let $L'$ be such a maximal submodule and consider the induced form $\langle\,\, ,\,\rangle': L'^{\perp}/L' \times L'^{\perp}/L' \to N$. If $\m
(L'^{\perp}/L') \neq 0$, then by Corollary \ref{ji} we can lift the non-trivial lower root of $L'^{\perp}/L'$ to obtain a module $L' \subsetneq L''$
with $L'' \subset
L''^{\perp}$, a contradiction. 
\end{proof}

We have the following lemma, of which the proof is very technical. 

\begin{lemma} \label{sir}
Assume that $\bi$ is non-degenerate. Suppose that $r=2$ and that there exists $x \in M \setminus M[\m]$ with $\pi \langle x, x \rangle=0$. Then there exists $L \subset M$ with $L \neq \mathrm{lr}(M)$ and $\m L^{\perp} \subseteq L \subseteq L^{\perp}$. 
\end{lemma}
\begin{proof}
We assume that $N=R$. We obtain a non-degenerate symmetric bilinear form $\bi_{\mathrm{even}}: M/M[\m] \times M/M[\m] \to R/R[\m]$ and hence there is $y \in M \setminus M[\m]$ such that $\pi \langle x,y \rangle \neq 0$. Now consider $H=Rx+Ry$. We first claim that $H \cong R/\m^{2} \oplus  R/\m^{2}$. Notice that $\m^2H=0$ and consider the following matrix: 
\begin{eqnarray*}
\left( \begin{array}{cc}
\langle x,x \rangle & \langle x,y \rangle \\
\langle y,x \rangle & \langle y,y \rangle 
\end{array} \right) =
\left( \begin{array}{cc}
\pi^{n-2} \cdot \pi r_1 & \pi^{n-2} \cdot r_2 \\
\pi^{n-2} \cdot r_2 & \pi^{n-2} \cdot r_3 
\end{array} \right)
\end{eqnarray*}
where $r_1, r_3 \in R$, $r_2 \in R^*$. Apply the determinant criterion (Lemma \ref{homc}) to see that the matrix would give a non-degenerate symmetric
bilinear form on $(R/\m^2)^2$. By Lemma \ref{sir0} we see that $H \cong \left( R/\m^2 \right)^2$.

Assume for the moment that $M=H$. Let $L'=R\pi x$. Then $L' \subseteq L'^{\perp}$. A short calculation show that $L'^{\perp}=Rx+R \pi y$, and
hence that $\m L'^{\perp} \subseteq L'$. Furthermore, $L' \neq \mathrm{lr}(H)$, because $\pi y \in \mathrm{lr}(H)$ but $\pi y \not \in L'$. 

For the general case write $M=H \perp H^{\perp}$ (Lemma \ref{orthogonal}) and let $L=L'\oplus \pi H^{\perp}$. Then $L \neq \mathrm{lr}(M)$ and it is
easy to see that $L \subseteq L^{\perp}$. Notice that $L^{\perp} \subseteq L'^{\perp}$ and an easy calculation shows that $\m L'^{\perp} \subseteq
L$. Hence $\m L^{\perp} \subseteq L$.
\end{proof}

\begin{lemma} \label{sir5}
Assume that $\bi$ is non-degenerate and that the only submodule $L \subseteq M$ satisfying $\m L^{\perp} \subseteq L \subseteq L^{\perp}$ is
$\mathrm{lr}(M)$. Suppose $M' \subseteq M$ satisfies $M' \subseteq M'^{\perp}$. Consider natural form $M'^{\perp}/M' \times M'^{\perp}/M' \to N$.
The following statements hold.
\begin{enumerate}
 \item 
 The only submodule $L' \subseteq M'^{\perp}/M'$ satisfying $\m L'^{\perp} \subseteq L' \subseteq L'^{\perp}$ is \\$\mathrm{lr}(M'^{\perp}/M')$;
 \item
$M' \subseteq \mathrm{lr}(M)$;
 \item
$\mathrm{lr}(M)/M'=\mathrm{lr}(M'^{\perp}/M')$.
\end{enumerate} 
\end{lemma}
\begin{proof}
There is a natural bijection between the set $\mathfrak{S}_1$ of submodules $M''/M' \subseteq M'^{\perp}/M'$ satisfying $\m \left( M''/M'
\right)^{\perp} \subseteq M''/M' \subseteq \left( M''/M' \right)^{\perp}$ and the set $\mathfrak{S}_2$ of submodules $L'' \subseteq M$ satisfying $M'
\subseteq L''$ and $\m L''^{\perp} \subseteq L'' \subseteq L''^{\perp}$.
Notice that $\# \mathfrak{S}_1=\# \mathfrak{S}_2 \geq 1$ (as $\mathrm{lr}(M'^{\perp}/M') \in \mathfrak{S}_2$) and the last set has size at most $1$ by
assumption. We conclude that both sets have size $1$ and contain only the lower root of $M$ respectively $M'^{\perp}/M'$. This also shows that $M'
\subseteq \mathrm{lr}(M)$ and it
gives $\mathrm{lr}(M)/M'=\mathrm{lr}(M'^{\perp}/M')$.  
\end{proof}

\begin{lemma} \label{sir4}
Assume that $\bi$ is non-degenerate. Let $x \in M$ with $\mathrm{Ann}_R(x)=\m^s$. Then for every $r' \in \m^{n-s}$ there is $y \in M$ with $\langle x,y \rangle=r'$. 
\end{lemma}
\begin{proof}
First consider the map 
\begin{eqnarray*}
\varphi: Rx &\to& \m^{n-s} \subseteq R \\
x &\mapsto& r'
\end{eqnarray*}
which is defined by assumption. As $R$ is an injective $R$-module (Corollary \ref{inj}), we can extend $\varphi$ to a map $\psi': M \to R$. Since $\langle\,\, ,\,\rangle$ is non-degenerate, we see that there is $y \in M$ with $\psi'(z)=\langle z,y \rangle$ for all $z \in M$. Hence we have $r'=\psi'(x)= \langle x,y \rangle$. 
\end{proof}

\begin{lemma} \label{hensel}
Assume that $\mathrm{char}(R/\m) \neq 2$ and that $N=R$. Let $x,y \in M$ be such that $R\langle x,x \rangle=\m^i$, $R \langle x,y \rangle=\m^j$ and
$R \langle y,y \rangle=\m^{k}$ where $i+k>2j$ and $i \geq j$. Then there exists $c \in \m^{i-j}$ such that the element $z=x+cy$ satisfies $
\langle z,z \rangle=0$ and $R \langle y,z \rangle=\m^j$. 
\end{lemma}
\begin{proof}
We give a proof by induction on $n-i$. If $n=i$, then we are directly done. Now continue with induction.
Let $\langle x, x \rangle=\pi^i r_1$ and let $\langle x, y \rangle=\pi^j r_2$ where $r_1, r_2 \in R^*$. Now let $c= \frac{- \pi^{i-j} r_1}{2 r_2}$ and
let $z'=x+cy$. We calculate:
\begin{eqnarray*}
\langle z',z' \rangle &=& \langle x+cy, x+cy \rangle \\
&=& \langle x, x \rangle +2c \langle x,y \rangle + c^2 \langle y, y \rangle \\
&=& \pi^i r_1+2 \frac{- \pi^{i-j} r_1}{2 r_2} \pi^j r_2 + c^2 \langle y, y \rangle \\
&=&  c^2 \langle y, y \rangle.
\end{eqnarray*}
Notice that $R c^2 \langle y, y \rangle=\m^l$ where $l=2(i-j)+k>i$.  
As $i-j+k>j$ we conclude
\begin{eqnarray*}
R \langle y, z' \rangle &=& R \langle y, x+cy \rangle \\
&=& R(\langle x,y \rangle + c \langle y,y \rangle) \\
&=& R(\langle x,y \rangle)=\m^j.
\end{eqnarray*}
These $z'$ and $y$ still satisfy the assumption of the lemma, since $l+k>i+k>2j$ and $l>j$. Notice that $\langle z',z' \rangle=\m^l$
where $l=2(i-j)+k>i$, so $n-l<n-i$ and we can apply our induction hypothesis to finish the proof. 
\end{proof}

\section{Proof of the equivalence} \label{pr}

In this section we will prove Theorem \ref{main1}.

\begin{proof}[Proof of Theorem \ref{main1}]

We first use Lemma \ref{ds} and Lemma \ref{ds2} to reduce to the case where $R$ is a local
artinian principal ideal ring. So assume that we are in the same situation as in the previous section and let $(R,\m)$ be a local
artinian principal ideal ring. Let $n$ be the length of $R$ as an $R$-module, and let $\m=(\pi)$. 

i $\implies$ ii: Recall that anisotropic vector spaces are non-degenerate. Hence the non-degeneracy follows from Theorem \ref{moma}. We will continue
by induction on $r$. If $r=0$ the statement follows directly. If $r=1$ we have $\langle\,\, ,\,\rangle=\langle\,\, ,\,\rangle_{\mathrm{odd}}$, which
is
anisotropic and as $\mathrm{lr}(M)=0$ the statement holds. Now continue with induction and suppose that $r \geq 2$. By Lemma \ref{sir3} it follows
that for any $L \subseteq M$ with $\m L^{\perp} \subseteq L \subseteq L^{\perp}$ we have $\m^{r-1}M \subseteq L$. Let $\varphi: M[\m^{r-1}] \to
\mathrm{Sh}(M)$ be the natural map. Now use Lemma \ref{sat} and the induction hypothesis on $\mathrm{Sh}(M)$ (use Lemma \ref{threes} to see that i
still holds) to conclude that $L=\varphi^{-1}(\mathrm{lr}(\mathrm{Sh}(M)))=\mathrm{lr}(M)$ (Lemma \ref{trivi}).

ii $\implies$ i: 
We will show that not i implies not ii. Suppose that i doesn't hold. If $\langle\,\, ,\,\rangle_{\mathrm{odd}}$ is isotropic, then we find $x \in
\mathrm{ur}(M) \setminus \mathrm{lr}(M)$ with $\langle x,x \rangle=0$. Apply Lemma \ref{killit} to $L=Rx$ to find a contradiction with ii. 

Now suppose that $\langle\,\, ,\,\rangle_{\mathrm{even}}$ is isotropic. We will show by induction on $r$ that we can find $L \subseteq M$ with $\m
L^{\perp} \subseteq L \subseteq L^{\perp}$, but $L \neq \mathrm{lr}(M)$. Remark that this can only happen if $r \geq 2$, since for $r=0, 1$ we have
$M/M[\m]=0$ and $\langle\,\, ,\,\rangle_{\mathrm{even}}$ is anisotropic. 

Suppose $r=2$. Then one can easily see that the assumptions of Lemma \ref{sir} are satisfied and hence there
exists $L \subseteq M$ with $\m L^{\perp} \subseteq L \subseteq L^{\perp}$, but $L \neq \mathrm{lr}(M)$.

We now continue by induction. Suppose that $r \geq 3$. Suppose that $x$ is an isotropic element of $\langle\,\, ,\,\rangle_{\mathrm{even}}$. Then $x$
gives an isotropic element in $\mathrm{Sh}(M)_{\mathrm{even}}$ (Lemma \ref{threes}, Lemma \ref{soe} and Definition \ref{goh}). As $\mathrm{Sh}(M)$ has
smaller exponent, we apply our induction hypothesis and we find $L' \subseteq \mathrm{Sh}(M)=M[\m^{r-1}]/\m^{r-1}M$ with $\m L'^{\perp} \subseteq L'
\subseteq L'^{\perp}$ and $L' \neq \mathrm{lr}(\mathrm{Sh}(M))$. Let $\varphi: M[\m^{r-1}] \to \mathrm{Sh}(M)$ be the canonical map, and let
$L=\varphi^{-1}(L')$. We see that $\m L^{\perp} \subseteq L \subseteq L^{\perp}$ (Lemma \ref{sat}) and $L \neq \mathrm{lr}(M)$. This contradicts ii.

ii $\implies$ iii: This is Lemma \ref{sir5}. 

iii $\implies$ ii: Suppose that $L \subseteq M$ satisfies $\m L^{\perp} \subseteq L \subseteq L^{\perp}$. Then iii gives $L \subseteq \mathrm{lr}(M)$ and $\mathrm{lr}(M)/L=\mathrm{lr}(L^{\perp}/L)=L/L$ (as $\m (L^{\perp}/L)=0$). It follows that $\mathrm{lr}(M)=L$ and we are done.

iii $\implies$ iv: Obvious. 

iv $\iff$ v: This is obvious. 

v $\implies$ ii if $\mathrm{char}(R/\m) \neq 2$: We will assume that $N=R$. We will give a proof by induction on the exponent of $M$ using shaving. Assume that $\m L^{\perp} \subseteq L \subseteq L^{\perp}$. By v we know that $L \subseteq \mathrm{lr}(M)$. We need to prove that $L=\mathrm{lr}(M)$.  

If $r=0,1$ we see that $L \subseteq \mathrm{lr}(M)=0$ and we are done. 

Now suppose that $r \geq 2$. As $L \subseteq \mathrm{lr}(M)$ it follows that $\mathrm{Ann}_R(L)=\m^i$ where $i \leq \lfloor \frac{r}{2} \rfloor <r$.
Let $\varphi: M[\m^{r-1}] \to \mathrm{Sh}(M)$ and consider $L'=\varphi(L)$, which by our induction hypothesis (and Lemma \ref{sat}) satisfies
$L'=\mathrm{lr}(\mathrm{Sh}(M))$. By Lemma \ref{trivi} we conclude that $L+\m^{r-1}M=\mathrm{lr}(M)$. Hence it is enough to prove that $\m^{r-1}M
\subseteq L$, or equivalently, $L^{\perp} \subseteq M[\m^{r-1}]$. Let $x \in L^{\perp}$, but $x \not \in M[\m^{r-1}]$. By assumption we know $\pi x
\in L$ and hence $0=\langle x,\pi x \rangle =\pi \langle x,x \rangle$, that is, $\langle x,x \rangle \in R[\m]$. Write $\langle x,x
\rangle=\pi^{n-1}r$ for some $r \in R$. By Lemma \ref{sir4} we can find $y \in M$ with $\langle x,y \rangle=\pi^{n-r}$. We can now apply Lemma
\ref{hensel} (here $i \geq n-1$, $j=n-r$ and $k \geq n-r$; we use that $r>1$ here) and we see that there is $c \in \m^{r-1}$ such that $z=x+cy$
satisfies $\langle z,z \rangle=0$. By our assumption in v we have $z \in \mathrm{lr}(M) \subseteq M[\m^{r-1}]$. Notice that $\pi cy=0$, and as $r \geq
2$ we find $cy \in M[\m^{r-1}]$. Hence we have $x=z-cy \in M[\m^{r-1}]$, a contradiction. This shows that $\mathrm{lr}(M)=L$ and hence we are done.

\end{proof}

\section{Quasi-anisotropy}

Let $R$ be an artinian principal ideal ring and let $M$ be a finitely
generated $R$-module. Let $N$ be an $R$-module such that $N \cong_R R$ and let $\bi: M \times M \to N$ be a symmetric $R$-bilinear form. In
this section we will define the concept of quasi-anisotropy for such a form $\bi$. Quasi-anistropy is a concept which doesn't give anything
interesting in the case where $R$ is a field. See \cite{KO1} and \cite{KO2} for the applications. 

\begin{definition} \label{sern}
First assume that $(R,\m)$ is local. Then $\bi$ is called \emph{quasi-anisotropic} if it is non-degenerate and both $\perp_{i \mathrm{\ even}} \rho_i(M)$ and $\perp_{i \mathrm{\ odd}, i \neq 1} \rho_i(M)$ are anisotropic (see Lemma \ref{goh}). 
Now assume that $R$ is an artinian principal ideal ring. Then $\bi$ is called \emph{quasi-anisotropic} if for all $\pa \in \mathrm{Spec}(R)$ the forms
$\bi_{\pa}: M_{\pa} \times M_{\pa} \to N_{\pa}$ (Lemma \ref{ds2}) are quasi-anisotropic. 
\end{definition}

\begin{remark}
If $\langle\,\, ,\,\rangle$ is anisotropic, it is automatically quasi-anisotropic by definition. For quasi-anisotropy we basically forget the
semi-simple part of $M$.
\end{remark}

\begin{definition}
We define $\mathrm{Soc}_R(M)$, the \emph{socle} of $M$, to be the sum of the simple submodules of $M$. Notice that this is a submodule of $M$.
\end{definition}

We will now give a couple of equivalent definitions of quasi-anistropy. 

\begin{theorem} \label{ksi}
Assume that $\bi$ is non-degenerate. Then the following statements are equivalent.  
\begin{enumerate}
\item
The form $\langle\,\, ,\,\rangle$ is quasi-anisotropic. 
\item
The induced form $\langle\,\, ,\,\rangle': M/\mathrm{Soc}_R(M) \times M/\mathrm{Soc}_R(M) \to R/\mathrm{Soc}_R(R)$ is anisotropic.
\item
For any $L \subseteq \mathrm{lr}(M)$ we have $\mathrm{lr}(L^{\perp}/L)=\mathrm{lr}(M)/L$.
\end{enumerate}
\end{theorem}

We first prove the following lemma. 

\begin{lemma} \label{paro2}
Assume that $\bi$ is quasi-anisotropic. Let $L \subseteq \mathrm{lr}(M)$ be a submodule. Then we have:
\begin{enumerate}
\item
$L \subseteq L^{\perp}$;
\item
$L^{\perp}/L$ is quasi-anisotropic;
\item
$\mathrm{lr}(L^{\perp}/L)=\mathrm{lr}(M)/L$.
\end{enumerate}
\end{lemma}
\begin{proof}
We may assume that $(R,\m)$ is local. 
We have $L \subseteq \mathrm{lr}(M) \subseteq \mathrm{ur}(M) \subseteq L^{\perp}$. Write $M=M_1 \perp M'$ where $M_1$ is free over $R/\m$ and
$\rho_1(M')=0$ (Theorem \ref{moma}). As $\mathrm{lr}(M_1)=0$ we can write $L=\{0\} \perp L'$. Then $L^{\perp}=M_1 \perp
L'^{\perp}$ (where $L'^{\perp} \subseteq M'$). Notice that by construction $M'$ is anisotropic. From Lemma \ref{sir5}
we see that $L'^{\perp}/L'$ is anisotropic and $\mathrm{lr}(M')/L'=\mathrm{lr}(L'^{\perp}/L')$. Hence $L^{\perp}/L=M_1 \perp L'^{\perp}/L'$ is
quasi-anisotropic. We then find
\begin{eqnarray*}
\mathrm{lr}(L^{\perp}/L) &=& \mathrm{lr}(M_1/0) \oplus \mathrm{lr}(L'^{\perp}/L') \\
&=& 0 \oplus \mathrm{lr}(M')/L' \\
&=& \mathrm{lr}(M)/L.
\end{eqnarray*}
\end{proof}

\begin{proof}[Proof of Theorem \ref{ksi}]
We again assume that $(R,\m)$ is local. In this case $\mathrm{Soc}_R(M)=M[\m]$ and $\mathrm{Soc}_R(R)=R[\m]$. Let $\pi \in R$ such that $(\pi)=\m$. 

i $\Longleftrightarrow$ ii: This directly follows from the fact that $\rho_i(M/M[\m]) = \rho_{i+1}(M)$ for $i \geq 1$, and hence we just lose
$\rho_1(M)$. Now check the definition.

i $\implies$ iii: This is Lemma \ref{paro2}.

iii $\implies$ i:  Suppose that $\perp_{i \mathrm{\ even}} \rho_i(M)$ or $\perp_{i>1 \mathrm{\ odd}} \rho_i(M)$ is isotropic, so automatically $r \geq
2$ (where $\mathrm{Ann}_R(M)=\m^r$). We will find a module $L \subseteq \mathrm{lr}(M)$ with $\mathrm{lr}(L^{\perp}/L) \neq \mathrm{lr}(M)/L$. First
assume that $M$ is homogeneous, say $M \cong (R/\m^r)^s$. Choose $x \in M \setminus \m M$ with $\langle x, \pi^{r-1}x \rangle=0$ (that such an $x$
exists is left to the reader). Consider the submodule $L=R \pi^{r-1} x \subseteq \mathrm{lr}(M)$ ($r \geq 2$ needed), which
satisfies $L \subseteq L^{\perp}$. Then a simple calculation, using Theorem \ref{ar1}, gives
\begin{eqnarray*}
\mathrm{length}_R \left( \mathrm{lr}(L^{\perp}/L) \right) &=& 2 \left\lfloor \frac{r-1}{2} \right\rfloor +(s-2) \left\lfloor \frac{r}{2}
\right\rfloor 
\end{eqnarray*}
and 
\begin{eqnarray*}
\mathrm{length}_R(\mathrm{lr}(M)/L) &=& s \left\lfloor \frac{r}{2} \right\rfloor-1.
\end{eqnarray*} 
The difference of these lengths is $2 \left(\lfloor \frac{r-1}{2} \rfloor - \lfloor \frac{r}{2} \rfloor \right)+1 \neq 0$. Hence we are done in
the homogeneous case. 

Now we will do the general case. Write $M=M_1 \perp \ldots \perp M_n$ as in Theorem \ref{moma}. If some $\rho_i(M)$ is isotropic, then again there is $x \in M_i$ with $\langle x, \pi^{i-1} x \rangle=0$ and we can consider $L=R\pi^{i-1} x$ and as above we contradict iii. 

We will give a proof by induction on $r$. If $r=2,3$ then we know that either $\rho_2(M)$ or $\rho_3(M)$ is isotropic, and we have considered this
case. Assume that $r \geq 4$ and let $\varphi: M[\m^{r-1}] \to M[\m^{r-1}]/\m^{r-1}M=\mathrm{Sh}(M)$ be the natural surjection. By our induction
hypothesis, in combination with Lemma \ref{threes}, we know that there is $L \subseteq \mathrm{lr}(\mathrm{Sh}(M))$ with $\mathrm{lr}(L^{\perp}/L)
\neq \mathrm{lr}(\mathrm{Sh}(M))/L$. By Lemma \ref{trivi} we have $\varphi^{-1}(\mathrm{lr}(\mathrm{Sh}(M)))=\mathrm{lr}(M)$. Hence $\varphi^{-1}(L)
\subseteq \mathrm{lr}(M)$.  
We now have
\begin{eqnarray*}
\mathrm{lr}(L^{\perp}/L) &\cong& \mathrm{lr}( \varphi^{-1}(L)^{\perp}/\varphi^{-1}(L))
\end{eqnarray*}
and
\begin{eqnarray*}
\mathrm{lr}(\mathrm{Sh}(M))/L &\cong& \mathrm{lr}(M)/\varphi^{-1}(L).
\end{eqnarray*}
As these maps are all natural, we find $\mathrm{lr}( \varphi^{-1}(L^{\perp})/\varphi^{-1}(L)) \neq \mathrm{lr}(M)/\varphi^{-1}(L)$ and this finishes
our proof. 
\end{proof}

\section{Determining the radical root}

Let $R$ be an artinian principal ideal ring and let $M$ be a finitely
generated $R$-module. Let $N$ be an $R$-module such that $N \cong_R R$ and let $\bi: M \times M \to N$ be a non-degenerate symmetric $R$-bilinear
form.
Recall the definition of the radical root of $(M,\bi)$ (Definition \ref{rad}):
\begin{eqnarray*}
\mathrm{rr}(M,\bi)=\mathrm{rr}(M)=\bigcap_{L \subseteq M: L \subseteq L^{\perp},\ L^{\perp}/L \mathrm{\ semisimple}} L.
\end{eqnarray*}
In this section we will discuss how one can determine this radical root.

Notice that we have the following formula, which implicitly uses Lemma
\ref{ds2}:
\begin{eqnarray*}
 \mathrm{rr}(M)=\bigoplus_{\pa \in \mathrm{Spec}(R)}\mathrm{rr}(M_{\pa}).
\end{eqnarray*}
For simplicity we will only study the case when $(R,\m)$ is local. The general case follows from the above formula.  

\begin{definition}
For $s \in \Z$ we define 
\begin{eqnarray*}
\mathrm{lr}_s(M)=\sum_{i \in \Z_{\geq 0}} \m^i M \cap M[\m^{i-s+1}]
\end{eqnarray*}
where we define $M[\m^{-t}]=0$ if $t \geq 0$. Remark that by definition we have $\mathrm{lr}_1(M)=\mathrm{lr}(M)$. 
\end{definition}

\begin{remark}
Remark that the above definition makes sense for any finitely generated $R$-module $M$ and furthermore that no form is needed in the definition.
Remark that $\mathrm{lr}_s(M)
\supseteq \mathrm{lr}_{s'}(M)$ if $s \leq s'$. Morover we see directly from the definition that $\mathrm{lr}_{s}(M)/\mathrm{lr}_{s-1}(M)$ is an
$R/\m$-module. If $\mathrm{Ann}(M)=\m^r$, we have $M=\mathrm{lr}_{-r+1}(M)$ and $\mathrm{lr}_r(M)=0$. Also notice that $\mathrm{lr}_s$ commutes with
direct sums. 
\end{remark}

The main theorem of this section is the following. Its proof will be given later in this section. 	

\begin{theorem} \label{srt}
Let $(R,\m)$ be local. Let $d \in \Z_{\geq 2}$ be minimal such that $\perp_{i \geq d \mathrm{\ even}} \rho_i(M)$ and
$\perp_{i \geq d \mathrm{\ odd}} \rho_i(M)$ are anisotropic over $R/\m$. The we have $\mathrm{rr}(M) \supseteq \mathrm{lr}_{d-1}(M)$.
\end{theorem}

The following lemma gives a more explicit description of the $\mathrm{lr}_s(M)$. 

\begin{lemma} \label{toto}
 Suppose that $(R,\m)$ is local and let $s \in \Z$. Then the following hold.
\begin{enumerate}
 \item 
 Suppose that $M$ is cyclic of length $r$, then we have
\begin{eqnarray*}
 \mathrm{length}_R(\mathrm{lr}_s(M))= \left\{ \begin{array}{cc}
                                              \lfloor \frac{r-s+1}{2} \rfloor & \mathrm{if\ } 0 \leq \lfloor \frac{r-s+1}{2} \rfloor \leq r \\
					      0	& \mathrm{if\ } \lfloor \frac{r-s+1}{2} \rfloor  \leq 0 \\
					      r & \mathrm{if\ } \lfloor \frac{r-s+1}{2} \rfloor  \geq r.
                                              \end{array}\right.
\end{eqnarray*}

 \item We have $\mathrm{lr}_s(M)^{\perp}=\mathrm{lr}_{1-s}(M)$.
\end{enumerate}
\end{lemma}
\begin{proof}
We prove the first statement. Let $t=\mathrm{min}(r-i,i-s+1)$. Notice that $\m^i M \cap M[\m^{i-s+1}]$ is the unique submodule of $M$ of length $t$ 
if $0 \leq t \leq r$, of length $0$ if $t \leq 0$ and of lenght $r$ if $t \geq r$. Hence we have
\begin{eqnarray*}
\mathrm{length}(\mathrm{lr}_{s}(M)) &=& \mathrm{min} \left( \mathrm{max} \left( \{\mathrm{min}(r-i,i-s+1): i \in \Z_{\geq 0}\} \cup \{0\} \right)
\cup \{r\} \right) \\
&=&  \mathrm{min} \left(\mathrm{max} \left( \{\mathrm{min}(r+s-1-i,i)-(s-1): i \in \Z_{\geq 0}\} \cup \{0\} \right) \cup \{r\} \right)\\
&=&  \mathrm{min} \left(\mathrm{max}(0, \lfloor \frac{r+s-1}{2} \rfloor-(s-1)) \cup \{r\} \right)\\
&=&  \mathrm{min} \left(\mathrm{max}(0, \lfloor \frac{r-s+1}{2} \rfloor) \cup \{r\} \right).
\end{eqnarray*} 
This finishes the proof of the first part.

We leave the proof of the second statement, which won't be used anywhere else in this article, to the reader. 
\end{proof}

The main ingredient in the proof of Theorem \ref{srt} is the following lemma. 

\begin{lemma} \label{serp}
Let $(R,\m)$ be local and let $\mathrm{Ann}_R(M)=\m^r$ (where $0 \leq r \leq \mathrm{length}_R(R)$). If $r \geq 2$ we let $\varphi: M[\m^{r-1}] \to
M[\m^{r-1}]/\m^{r-1}M=\mathrm{Sh}(M)$ be the natural map. 
Then we have
\begin{eqnarray*}
\mathrm{rr}(M)= \left\{ \begin{array}{cc}
0 & \mathrm{if\ }r \leq 1; \\
\varphi^{-1}\left(\mathrm{rr}(\mathrm{Sh}(M))\right) & \mathrm{if\ }r \geq 2 \mathrm{\ and\ } \rho_r(M) \mathrm{\ is\ anisotropic}.
\end{array} \right.
\end{eqnarray*}
\end{lemma}
\begin{proof}
The first case is obvious as $\mathrm{lr}(M)=0$ in that case. The second case follows from Lemma \ref{sir3} and Lemma \ref{sat}. 
\end{proof}

\begin{remark} \label{hos}
 A technical proof (Theorem 5.5.2 \cite{KO1}) shows that in the previous lemma we have that $\mathrm{rr}(M)=0$ if $r \geq 2$, $\rho_r(M)$ is
isotropic and $\mathrm{char}(R/\m) \neq 2$. If $\mathrm{char}(R/\m) \neq 2$, this is not necessarily true. 
\end{remark}

We can finally prove Theorem \ref{srt}.

\begin{proof}[Proof of Theorem \ref{srt}]
Let $\mathrm{Ann}_R(M)=\m^r$ where $0 \leq r \leq \mathrm{length}_R(R)$ and write $M=\bigoplus_{i: 1 \leq i \leq r} M_i$ where $M_i \cong
(R/\m^i)^{s_i}$ for some $s_i \in \Z_{\geq 0}$. We give a proof by induction on $r$. If $r \geq 2$ we let $\varphi: M[\m^{r-1}] \to
M[\m^{r-1}]/\m^{r-1}M=\mathrm{Sh}(M)$ be the natural map. Let $M'=\bigoplus_{i:\ d \leq i \leq r} \m^{\lfloor \frac{i+d-1}{2} \rfloor}
 M_i$. We first claim that $M'=\mathrm{lr}_{d-1}(M)$. Indeed, if $M$ is cyclic of length $r$ we have 
\begin{eqnarray*}
\mathrm{length}( \m^{\lfloor \frac{i+d-1}{2} \rfloor} M) &=& \mathrm{max}(0,r-\lfloor \frac{r+d-1}{2} \rfloor) = \mathrm{max}(0, \lfloor
\frac{r-d+2}{2} \rfloor) \\
&=& \mathrm{length}(\mathrm{lr}_{d-1}(M))
\end{eqnarray*}
by Lemma \ref{toto} and the general case follows easily.

Remark that if $d>r$, we need to prove that $\mathrm{rr}(M) \supseteq 0$, which is obviously correct. This already shows that the formula is correct
for $r=0, 1$. Assume
now that
$d \leq r$. If $r=2$ and $d=2$ notice that $M'=\m
M_2=\varphi^{-1}(0)=\varphi^{-1}(\mathrm{rr}(\mathrm{Sh}(M)))=\mathrm{rr}(M)$, according to Lemma \ref{serp}. 

Now let $r>2$. First remark that the $d$ in the statement of this theorem doesn't change when passing from $M$ to $\mathrm{Sh}(M)$ (Lemma
\ref{threes}). Now write $\mathrm{Sh}(M)=
\bigoplus_{i:\ 1 \leq i \leq r-1} M_i'$ where $M_i'=M_i$ for $i \neq r-2$ and $M_{r-2}'=M_{r-2} \oplus M_{r}[\m^{r-1}]/\m^{r-1}M_r$. We then have,
according to Lemma \ref{serp} and the induction hypothesis
\begin{eqnarray*}
 \mathrm{rr}(M) &\supseteq& \varphi^{-1}(\mathrm{rr}(\mathrm{Sh}(M))) = \varphi^{-1} \left(\bigoplus_{i:\ d \leq i \leq r-1} \m^{\lfloor
\frac{i+d-1}{2} \rfloor} M_i'\right) \\
&=& \varphi^{-1} \left( \left( \bigoplus_{i:\ d \leq i \leq r-1} \m^{\lfloor \frac{i+d-1}{2} \rfloor} M_i \right) \oplus \m^{\lfloor
\frac{r-2+d-1}{2} \rfloor} M_{r}[\m^{r-1}]/\m^{r-1}M_{r}\right) \\
&=& \left( \bigoplus_{i:\ d \leq i \leq r-1} \m^{\lfloor \frac{i+d-1}{2} \rfloor} M_i \right) \oplus \m^{\lfloor
\frac{r-2+d-1}{2} \rfloor+1} M_r \\
&=& \bigoplus_{i:\ d \leq i \leq r} \m^{\lfloor \frac{i+d-1}{2} \rfloor} M_i = \mathrm{lr}_{d-1}(M).
\end{eqnarray*}

\end{proof}

\begin{remark} \label{sus}
 It follows from Remark \ref{hos} that we have equality instead of $\supseteq$ in Theorem \ref{srt} if $\mathrm{char}(R/\m) \neq
2$. 
\end{remark}

\begin{proposition} \label{us}
Assume that $\bi$ is quasi-anisotropic. Then $\mathrm{rr}(M)=\mathrm{lr}(M)$. 
\end{proposition}
\begin{proof}
For the proof we may assume that $R$ is local and that we are in the case of Corollary \ref{srt}. Here we have $d=2$ and we get
\begin{eqnarray*}
\mathrm{lr}(M) \supseteq \mathrm{rr}(M) \supseteq \mathrm{lr}_0(M)=\mathrm{lr}(M).
\end{eqnarray*}
Hence we find $\mathrm{rr}(M)=\mathrm{lr}(M)$.
\end{proof}

\begin{remark}
 Remark \ref{sus} shows that the converse of Lemma \ref{us} is also true if $2$ is a unit in $R/\mathrm{Ann}_R(M)$.
\end{remark}

\section{Acknowledgements} 

I would like to thank Professor Hendrik Lenstra for essentially coming up with all the new theory and for helping me to write this article. Without
his help this article wouldn't be possible.


\begin{thebibliography}{XX}

\bibitem{AT} M.F. Atiyah, I.G. MacDonald, Introduction to Commutative Algebra, Addison-Wesley Publishing Company, 1969
\bibitem{KO1} M. Kosters, Anisotropic modules and the integral closure, Universiteit Leiden, 2010,
\url{http://www.math.leidenuniv.nl/scripties/KostersMaster.pdf}	
\bibitem{KO2} M. Kosters, Anisotropy and the integral closure, to be submitted
\bibitem{LA} S. Lang, Algebra, Revised third edition, Springer-Verlag, 2002   
\bibitem{LAM} T.Y. Lam, Lectures on Modules and Rings, Springer-Verlag, 1999
\bibitem{LAM2} T.Y. Lam, The Algebraic Theory of Quadratic Forms, W. A. Benjamin, 1973
\end{thebibliography}
\end{document}